\documentclass[]{article}
\usepackage{}
\usepackage{amsfonts}
\usepackage{latexsym,bm}
\usepackage{amssymb}
\usepackage{amsmath}
\usepackage{amsthm}
\usepackage{mathrsfs}
\usepackage[all]{xy}
\usepackage{colortbl,caption}
\usepackage{bbm}

\theoremstyle{plain}
\newtheorem{thm}{Theorem}

\newtheorem{prop}[thm]{Proposition}
\newtheorem{cor}[thm]{Corollary}
\newtheorem{lem-defi}[thm]{Lemma-Definition}

\theoremstyle{remark}

\newtheorem{e.g.}[thm]{Example}
\theoremstyle{definition}

\newcommand{\mcB}{{\mathcal B}}

\newcommand{\mcL}{{\mathcal L}}

\newcommand{\mcO}{{\mathcal O}}
\newcommand{\mcP}{{\mathcal P}}

\newcommand{\mcV}{{\mathcal V}}
\newcommand{\mcW}{{\mathcal W}}
\newcommand{\mcX}{{\mathcal X}}
\newcommand{\mcY}{{\mathcal Y}}

\newcommand{\C}{{\mathbb C}}

\newcommand{\mbP}{{\mathbb P}}

\newcommand{\mbR}{{\mathbb R}}

\newcommand{\mbZ}{{\mathbb Z}}

\newcommand{\Pic}{\text{Pic}}

\begin{document}

\title{Deformation of product of complex Fano manifolds}
\author{Qifeng Li}
\date{}
\maketitle

\begin{abstract}
Let $\mcX$ be a connected family of complex Fano manifolds. We show that if some fiber is the product of two manifolds of lower dimensions, then so is every fiber. Combining with previous work of Hwang and Mok, this implies immediately that if a fiber is (possibly reducible) Hermitian symmetric space of compact type, then all fibers are isomorphic to the same variety.
\end{abstract}




\bigskip

We consider varieties defined over the field of complex numbers $\C$. The aim of this short note is to show the following result.

\begin{thm}\label{thm. deform Fano product}
Let $\pi: \mcX\rightarrow S\ni 0$ be a holomorphic map onto a connected complex manifold $S$ such that all fibers are connected Fano manifolds and $\mcX_0\cong\mcX'_0\times\mcX''_0$. Then there exist unique holomorphic maps
\begin{eqnarray*}
f': \mcX\rightarrow\mcX' & \pi': \mcX'\rightarrow S, \\
f'': \mcX\rightarrow\mcX'' & \pi'': \mcX''\rightarrow S,
\end{eqnarray*}
such that $\mcX'_0$ and $\mcX''_0$ coincide with the fibers of $\pi'$ and $\pi''$ at $0\in S$, the morphism $\pi=\pi'\circ f'=\pi''\circ f''$ and $\mcX_t\cong\mcX'_t\times\mcX''_t$ for all $t\in S$.
\end{thm}

It is necessary to assume each fiber $\mcX_t$ being Fano in Theorem \ref{thm. deform Fano product}. The quadric surface $\mbP^1\times\mbP^1$ can be deformed to Hirzebruch surface $\Sigma_m$ with $m$ even. On the other hand, a Hirzebruch surface $\Sigma_m$ with $m\neq 0$ is neither a Fano manifold nor a product of two curves.

The following vanishing result is standard, which helps to study relative Mori contraction in our setting. The later is important for our proof of Theorem \ref{thm. deform Fano product}.

\begin{prop}\label{prop. vanishing of cohomology}
Let $\pi: \mcX\rightarrow\Delta^n$ be a holomorphic map onto the unit disk $\Delta^n$ of dimension $n$ such that all fibers are connected Fano manifolds. Suppose $\mcL$ is a holomorphic line bundle on $\mcX$ such that $\mcL_t:=\mcL|_{\mcX_t}$ is nef on $\mcX_t$ for each $t\in\Delta^n$. Then $H^k(\mcX, \mcL)=0$ for all $k\geq 1$.
\end{prop}

\begin{proof}
Since each $\mcX_t$ is Fano, $H^k(\mcX_t, \mcL_t)=0$ and $R^k\pi_*\mcL=0$ for all $k\geq 1$. Then $H^k(\mcX, \mcL)=H^k(\Delta^n, \pi_*\mcL)=0$ for all $k\geq 1$ since $\Delta^n$ is a Stein manifold.
\end{proof}

We summarize several facts in Proposition \ref{prop. relative MMP for Fano deformation} on relative Mori contraction in our setting. One can consult \cite{KaMaMa87} or \cite{KoMo98} for standard notation and results in relative Minimal Model Program. By \cite[Theorem 1]{Wis09} the Mori cone of each $\mcX_t$ is invariant in our setting, which we restate as conclusion (ii) in Proposition \ref{prop. relative MMP for Fano deformation}.

\begin{prop}\label{prop. relative MMP for Fano deformation}
Let $\pi: \mcX\rightarrow\Delta^n$ be as in Proposition \ref{prop. vanishing of cohomology}. Then for each $t\in\Delta^n$ the following hold.

(i) The natural morphism $\text{Pic}(\mcX)\rightarrow\text{Pic}(\mcX_t)$ is an isomorphism.

(ii) There are natural identities as follows:
\begin{eqnarray}\label{eqn. NE(Xt)=NE(X over delta)}
N_{1}(\mcX_t)=N_{1}(\mcX/\Delta^n) \mbox{ and } \overline{NE}(\mcX_t)=\overline{NE}(\mcX/\Delta^n).
\end{eqnarray}

(iii) Denote by $\Phi: \mcX\rightarrow\mcY$ the relative Mori contraction over $\Delta^n$ associated to an extremal face $F\subset\overline{NE}(\mcX/\Delta^n)$. Then $\Phi_t$ is the Mori contraction associated to $F\subset \overline{NE}(\mcX_t)$ under identification \eqref{eqn. NE(Xt)=NE(X over delta)}. In particular $\mcY$ and $\mcY_t$ are normal varieties.
\end{prop}

\begin{proof}
(i) By Proposition \ref{prop. vanishing of cohomology} the map $e: \text{Pic}(\mcX)\rightarrow H^2(\mcX, \mbZ)$, induced by Euler exact sequence induces, is an isomorphism. Since $\mcX_t$ is a Fano manifold, the map $e_t: \text{Pic}(\mcX_t)\rightarrow H^2(\mcX_t, \mbZ)$ is an isomorphism. The natural map $r: H^2(\mcX, \mbZ)\rightarrow H^2(\mcX_t, \mbZ)$ is also an isomorphism since $\Delta^n$ is a contractible topology space. Then (i) follows from the fact that the natural morphism $\text{Pic}(\mcX)\rightarrow\text{Pic}(\mcX_t)$ coincides with the composition $e_t^{-1}\circ r\circ e$.


\medskip

(ii) The linear map $N^1(\mcX/\Delta^n)\rightarrow N^1(\mcX_t)$ is surjective by (i). Then we have injective linear maps
$$\gamma_t: N_1(\mcX_t)\rightarrow N_1(\mcX/\Delta^n) \mbox{ and } \gamma_t^+: \overline{NE(\mcX_t)}\rightarrow \overline{NE}(\mcX/\Delta^n).$$
By Theorem 1 in \cite{Wis09} the local identification $H_2(\mcX_0, \mbR)=H_2(\mcX_t, \mbR)$ yields the identity of Mori cones $\overline{NE}(\mcX_0)=\overline{NE}(\mcX_t)$. Since the local identification $H^2(\mcX_0, \mbR)=H^2(\mcX_t, \mbR)$ gives the identity of $\mcL|_{\mcX_0}$ and $\mcL|_{\mcX_t}$ for any $\mcL\in\Pic(\mcX)$, we know by (i) that
\begin{eqnarray}\label{eqn. NE(Xt)=NE(X0)}
\gamma_t^+(\overline{NE}(\mcX_t))=\gamma_0^+(\overline{NE}(\mcX_0)), & \gamma_t(N_1(\mcX_t))=\gamma_0(N_1(\mcX_0)).
\end{eqnarray}
The abelian group $Z_1(\mcX/\Delta^n)$ (resp. semigroup $Z_1^+(\mcX/\Delta^n)$), generated by reduced irreducible curves that are contracted by $\pi$, is isomorphic to $\bigoplus\limits_{s\in\Delta^n}Z_1(\mcX_s)$ (resp. $\bigoplus\limits_{s\in\Delta^n}Z_1^+(\mcX_s)$). Then \eqref{eqn. NE(Xt)=NE(X0)} implies that $\gamma_t$ and $\gamma_t^+$ are bijections.

\medskip

(iii) By Contraction Theorem there exist $\mcL\in\text{Pic}(\mcX)$ and a homomorphism $\Phi: \mcX\rightarrow\mcY:=\text{Proj}_{\Delta^n}R(\mcL)$ over $\Delta^n$ such that $\Phi_*\mcO_{\mcX}=\mcO_{\mcY}$ and
$$F=\{\eta\in \overline{NE}(\mcX/\Delta^n)\mid \deg\mcL|_\eta=0\},$$
where $R(\mcL):=\bigoplus\limits_{m\geq 0}H^0(\mcX, \mcL^{\otimes m})$. Then (iii) follows from (ii) and the claim that $H^0(\mcX, \mcL^{\otimes m})\rightarrow H^0(\mcX_t, \mcL^{\otimes m}_t)$ is surjective for all $m\geq 0$.



The Cartier divisor $\mcL^{\otimes m}\otimes\mcO_{\mcX}(-\mcX_t)$ is nef on $\mcX_s$ for all $s\in\Delta^n$, which implies that $H^1(\mcX, \mcL^{\otimes m}\otimes\mcO_{\mcX}(-\mcX_t))=0$ by Proposition \ref{prop. vanishing of cohomology}. Hence the homomorphism $H^0(\mcX, \mcL^{\otimes m})\rightarrow H^0(\mcX_t, \mcL^{\otimes m}_t)$, induced by the short exact sequence
$$0\rightarrow\mcL^{\otimes m}\otimes\mcO_{\mcX}(-\mcX_t)\rightarrow\mcL^{\otimes m}\rightarrow\mcL^{\otimes m}_t\rightarrow 0,$$
is surjective, verifying the claim.
\end{proof}

Local verification of Theorem \ref{thm. deform Fano product} is closely related to Kodaira's stability as follows.

\begin{prop}\cite[Theorem 4]{Kod63} \label{prop. Kodaria stability}
Let $p: W\rightarrow B$ be a smooth map between complex manifolds such that each fiber $F$ is a compact connected complex manifold with $H^1(F, \mcO_F)=0$, and let $\omega: \mcW\rightarrow S\ni 0$ be a smooth family of complex manifold with $\omega^{-1}(0)=W$. Then there exists a commutative diagram
\begin{eqnarray*}
\xymatrix{\mcW_U\ar[d]_-{\omega_U}\ar[r]^-{\mcP} & \mcB\ar[ld]^-{\phi} \\
U, &
}
\end{eqnarray*}
where $U$ is an analytic open neighborhood $U$ of $0\in S$, $\mcB$ is a complex manifold smooth over $U$, $\mcW_U:=\omega^{-1}(U)$, $\omega_U:=\omega|_{\mcW_U}$, and the fiber $\mcP_0: \mcW_0\rightarrow\mcB_0$ at $0\in U$ coincides with $p: W\rightarrow B$.
\end{prop}

Let us collect two results on split vector bundle from \cite{HM98} and \cite{Hor07}, which will be useful in our proof of Theorem \ref{thm. deform Fano product}.

\begin{prop}\cite[page 409]{HM98}\label{prop. extending direct summmand bundles}
Let $\mcV$ be a vector bundle over a connected complex manifold $X$. Suppose there is a complex subvariety $D\subset X$ and vector bundles $\mcV_1$ and $\mcV_2$ over $X\setminus D$ such that $\dim D \leq \dim X - 2$ and $\mcV|_{X\setminus D}=\mcV_1\oplus\mcV_2$. Then $\mcV_1$ and $\mcV_2$ can be extended uniquely as vector bundles $\mcV'_1$ and $\mcV'_2$ over $X$ such that $\mcV=\mcV'_1\oplus\mcV'_2$.
\end{prop}

\begin{proof}
We repeat briefly the corresponding argument on page 409 of \cite{HM98}, since no explicit statement is given there. The problem is local, and we can assume that $X=\Delta^n$ and $\mcV$ is the trivial vector bundle $X\times V$. Then there is a holomorphic map $\Phi: X\setminus D\rightarrow Y:=Gr(r_1, V)\times Gr(r_2, V)$ sending $x\in X\setminus D$ to $((\mcV_1)_x, (\mcV_2)_x)$, where $r_1$ and $r_2$ is the rank of $\mcV_1$ and $\mcV_2$ respectively. The closed subset $E:=\{(V_1, V_2)\in Y\mid V_1\cap V_2\neq 0\}$ is an ample prime divisor of $Y$. Since $\Phi(X\setminus D)$ is contained in the affine variety $Y\setminus E$, we can extend $\Phi$ to $X$ uniquely by Hartogs' Extension Theorem.
\end{proof}


\begin{prop}\cite[Theorem 1.4]{Hor07} \label{prop. RC-mfds with split tangent bundles}
Let $X$ be a rationally connected manifold such that $T_X=V_1\oplus V_2$. If $V_1$ or $V_2$ is integrable, then $X$ is isomorphic to a product $Y_1\times Y_2$ such that $V_j=p^*_{Y_j}T_{Y_j}$ for $j=1, 2$.
\end{prop}

\begin{proof}[Proof of Theorem \ref{thm. deform Fano product}]
The problem is local, and we can assume $S=\Delta^n\ni 0$. By Proposition \ref{prop. relative MMP for Fano deformation} we can extend projections $f'_0: \mcX_0\rightarrow\mcX'_0$ and $f''_0: \mcX_0\rightarrow\mcX''_0$ to relative Mori contractions $f':\mcX\rightarrow\mcX'$ and $f'':\mcX\rightarrow\mcX''$. Moreover for any $t\in\Delta^n$, fibers $\mcX'_t$ and $\mcX''_t$ are normal, and morphisms $f'_t$ and $f''_t$ are Mori contractions. Denote by $\pi': \mcX'\rightarrow\Delta^n$ and $\pi'': \mcX''\rightarrow\Delta^n$ the corresponding morphisms over $\Delta^n$.

By Proposition \ref{prop. Kodaria stability} and Proposition \ref{prop. relative MMP for Fano deformation}(ii), there exists an analytic open neighborhood $U$ of $0\in\Delta^n$ such that $\mcX'_U$ and $\mcX''_U$ are complex manifolds, and
\begin{eqnarray*}
f'_U: \mcX_U\rightarrow\mcX'_U & \pi'_U: \mcX'_U\rightarrow U, \\
f''_U: \mcX_U\rightarrow\mcX''_U & \pi''_U: \mcX''_U\rightarrow U,
\end{eqnarray*}
are holomorphic maps.
Then by generic smoothness property of $f'$ and $f''$ there is a finite subset $B\subset\Delta^n\setminus U$ such that $\mcX'_\Omega$ and $\mcX''_\Omega$ are complex manifolds, and
\begin{eqnarray*}
f'_\Omega: \mcX_\Omega\rightarrow\mcX'_\Omega & \pi'_\Omega: \mcX'_\Omega\rightarrow \Omega, \\
f''_\Omega: \mcX_\Omega\rightarrow\mcX''_\Omega & \pi''_\Omega: \mcX''_\Omega\rightarrow \Omega,
\end{eqnarray*}
are holomorphic maps, where $\Omega:=\Delta^n\setminus B$.

The morphism
\begin{eqnarray*}
f: \mcX\rightarrow \mcX'\times_{\Delta^n}\mcX''
\end{eqnarray*}
over $\Delta^n$ induced by $f'$ and $f''$ contracting no curves on $\mcX$. It follows that there exists a positively codimensional closed subset $D$ of $\mcX_B:=\pi^{-1}(B)$ such that the relative tangent bundles $T^{f'}$ and $T^{f''}$ are holomorphic on $\mcX\setminus D$ and
$$T^\pi|_{\mcX\setminus D}=V'\oplus V'',$$
where $V':=(T^{f'})|_{\mcX\setminus D}$ and $V'':=(T^{f''})|_{\mcX\setminus D}$.
By Proposition \ref{prop. extending direct summmand bundles} we can extend $V'$ and $V''$ uniquely as vector bundles $\widetilde{V}'$ and $\widetilde{V}''$ over $\mcX$ such that $T^\pi=\widetilde{V}'\oplus \widetilde{V}''$. The distributions $\widetilde{V}'$ and $\widetilde{V}''$ are both integrable, since
\begin{eqnarray}\label{eqn. tilde(V1)= T(Phi) on open part}
\widetilde{V}'|_{\mcX\setminus D}=V'=(T^{f'})|_{\mcX\setminus D}, & \widetilde{V}''|_{\mcX\setminus D}=V''=(T^{f''})|_{\mcX\setminus D}.
\end{eqnarray}
We can apply Proposition \ref{prop. RC-mfds with split tangent bundles} to $\mcX_t$ with $t\in\Delta^n$, since Fano manifolds are rationally connected. It follows that $\mcX_t\cong \mcX'_t\times\mcX''_t$ for each $t\in\Delta^n$.

The leaves of $\widetilde{V}'$ (resp. $\widetilde{V''}$) are all Fano manifolds and hence simply-connected. There is a submersion $\phi': \mcX\rightarrow \mcY'$ (resp. $\phi'': \mcX\rightarrow \mcY''$) to a smooth complex variety $\mcY'$ (resp. $\mcY''$) over $\Delta^n$ such that $\widetilde{V}'=T^{\phi'}$ (resp. $\widetilde{V}''=T^{\phi''}$), where the smoothness of $\mcY'$ (resp. $\mcY''$) at an arbitrary point follows from the the simple-connectedness of the corresponding $\widetilde{V}'$-leaf (resp. $\widetilde{V}''$-leaf), see \cite[Proposition 3.7]{Mol88}.
Then \eqref{eqn. tilde(V1)= T(Phi) on open part} implies that $\mcY'$, $\mcY''$, $\phi'$ and $\phi''$ coincide with $\mcX'$, $\mcX''$, $f'$ and $f''$ respectively. The conclusion follows.
\end{proof}

It is proved by Hwang and Mok \cite{HM98} that irreducible Hermitian symmetric spaces of compact type are rigid under K\"{a}hler deformation. Combining with Theorem \ref{thm. deform Fano product}, we get the following corollary immediately.

\begin{cor}
Let $\pi: \mcX\rightarrow S\ni 0$ be a smooth family of connected Fano manifolds over a connected complex manifold $S$. Suppose that $\mcX_0$ is a (possibly reducible) Hermitian symmetric space of compact type. Then we have $\mcX_t\cong\mcX_0$ for all $t\in S$.
\end{cor}

\textbf{\normalsize Acknowledgements.} I want to thank Professor Jun-Muk Hwang and Professor Baohua Fu for help related with this note. This work is supported by National Researcher Program of National Research Foundation of Korea (Grant No. 2010-0020413).

\bigskip

Korea Institute for Advanced Study, Seoul, Republic of Korea

qifengli@kias.re.kr


\begin{thebibliography}{6}

\bibitem{Hor07} H\"{o}ring, Andreas Uniruled varieties with split tangent bundle. Math. Z. 256 (2007), no. 3, 465-479.

\bibitem{HM98} Hwang, Jun-Muk; Mok, Ngaiming Rigidity of irreducible Hermitian symmetric spaces of the compact type under K\"{a}hler deformation. Invent. Math. 131 (1998), no. 2, 393-418.


\bibitem{KaMaMa87}  Kawamata, Yujiro; Matsuda, Katsumi; Matsuki, Kenji Introduction to the minimal model problem. Algebraic geometry, Sendai, 1985, 283-360, Adv. Stud. Pure Math., 10, North-Holland, Amsterdam, 1987.

\bibitem{Kod63} Kodaira, K.: On stability of compact submanifolds of complex manifolds. Am. J. Math. 85, 79-94 (1963).

\bibitem{KoMo98} Koll\'{a}r, J\'{a}nos; Mori, Shigefumi Birational geometry of algebraic varieties. With the collaboration of C. H. Clemens and A. Corti. Translated from the 1998 Japanese original. Cambridge Tracts in Mathematics, 134. Cambridge University Press, Cambridge, 1998.

\bibitem{Mol88} Molino, P.: Riemannian Foliations. Progress in Mathematics, vol. 73. Birkh\"{a}user Boston Inc., Boston, MA (1988).


\bibitem{Wis09} Wi\'{s}niewski, J. A. Rigidity of the Mori cone for Fano manifolds. Bulletin of the London Mathematical Society 41, no. 5 (2009): 779-81.

\end{thebibliography}
 \end{document}